\documentclass[framed]{customarticle}

\title{Stochastic Krasnoselskii-Mann Iterations: \\ Convergence without Uniformly Bounded Variance}
\author{
Daniel Cortild$^1$ \orcidCortild (\texttt{daniel.cortild@maths.ox.ac.uk})\\
Coralia Cartis$^1$ \orcidCartis (\texttt{coralia.cartis@maths.ox.ac.uk})\\
Juan Peypouquet$^2$ \orcidPeypouquet (\texttt{j.g.peypouquet@rug.nl}) \\ \\
$^1$Mathematical Institute, University of Oxford, Oxford, United Kingdom\\
$^2$Bernoulli Institute, University of Groningen, Groningen, Netherlands
}
\date{Last Compiled: \today}

\begin{document}

\maketitle

\begin{abstract}
    \noindent We investigate the Stochastic Krasnoselskii-Mann iterations for expected nonexpansive fixed-point problems in a real separable Hilbert space. We establish convergence guarantees under significantly weaker assumptions on the variance than those typically used in the literature. In particular, instead of a uniform bound on the variance of the stochastic oracle, we only assume finite variance at a single fixed point. Under this assumption, we prove almost sure weak convergence of the iterates, derive convergence rates for the expected residual and the last-iterate residual, and obtain almost sure convergence rates for the running minimum residual. Notably, we recover the best-known stochastic oracle complexity without imposing uniformly bounded variance. We illustrate the applicability of our results to Stochastic Gradient Descent, where we recover known guarantees, to Stochastic Three-Operator Splitting and Stochastic  Backward-Forward Splitting, for which we obtain the first results that avoid uniform variance bounds, and to a novel Stochastic Lifted Three-Operator Splitting.

    \vspace{1em}\noindent\textbf{Keywords.} Krasnoselskii-Mann Iterations, Variance-at-Solution Assumption, Stochastic Gradient Descent, Stochastic Three-Operator Splitting, Stochastic Backward-Forward Splitting. 

    \vspace{1em}\noindent\textbf{Mathematics Subject Classification (2000)} 46N10, 47J26, 65K10, 90C25, 90C15.
\end{abstract}
\section{Introduction}

Fixed-point problems are central in optimization and variational analysis, providing a unified and systematic approach to solve nonlinear and operator equations. In this work, we consider the fixed point problem 
\begin{equation*}
    \text{Find $p\in \Fix(T)$},
\end{equation*}
where $T\colon \H\to \H$ is a nonexpansive mapping on a real separable Hilbert space $\H$, meaning that $\|Tx-Ty\|\le \|x-y\|$ for all $x, y\in \H$, where $\H$ has associated inner product $\langle \cdot, \cdot \rangle$ and induced norm $\|\cdot\|$, and $\Fix(T)=\{x\in \mathcal H\colon Tx=x\}$ denotes its set of fixed points. A standard algorithm for tackling such problems is the Krasnoselskii-Mann iterations \citep{mann_mean_1953,krasnoselskii_two_1955}, which are given by
\begin{equation}\tag{KM}
    x_{k+1}=(1-\lambda_k)x_k+\lambda_kTx_k\quad \text{for $k\ge 0$},
\end{equation}
where $x_0\in \H$ and $(\lambda_k)\subset (0, 1)$ is a sequence of \textit{relaxation parameters}. This was introduced independently in \cite{mann_mean_1953} for $\lambda_k=\frac{1}{k+1}$ and in \cite{krasnoselskii_two_1955} for $\lambda_k\equiv 1/2$. Weak convergence of the iterates was later established in \cite{schaefer_uber_1957} for constant parameters and for variable parameters satisfying $\sum_{k=0}^\infty \lambda_k(1-\lambda_k)=+\infty$ in \cite{groetsch_note_1972}. For a recent overview of the method, we refer to \cite{dong_krasnoselskiimann_2022,cortild_krasnoselskii_2025} and the references therein.

A modern challenge resides within the computability of the operator $T$. It is in general computationally prohibitive to evaluate the operator at each iteration. As such, we assume access to the operator only through a stochastic oracle $T_\xi$. We assume this estimator to be unbiased, namely that $T=\ExpD{\xi\sim \D}{T_\xi}$, where $\D$ is a probability distribution parameterizing the operators, such that the problem becomes
\[
    \text{Find $p\in \Fix(T)=\Fix(\ExpD{\xi}{T_\xi})$}.
\]
To this end, we study a stochastic variant of the standard Krasnoselskii-Mann iterations, given in Algorithm \ref{alg:SKM}. As in stochastic optimization, the stochasticity arises from sampling one operator at each iteration, rather than evaluating the expected operator. Stochastic Krasnoselskii-Mann iterations were previously studied in a block-coordinate framework with additive stochastic errors in \cite{combettes_stochastic_2015,combettes_stochastic_2019}.

\begin{algorithm}[H]
\renewcommand{\thealgorithm}{SKM}
\caption{(Stochastic Krasnoselskii-Mann)}\label{alg:SKM}
\begin{algorithmic}
\Require $T_\xi\colon \H\to \H$ nonexpansive with $\ExpD{\xi}{T_\xi}=T$, Relaxation parameters $\lambda_k\in (0, 1)$.
\For{$k=0, \ldots, K-1$}
    \State Draw $\xi_k\sim \mathcal D$ randomly and independently
    \State Update $x_{k+1}=(1-\lambda_k)x_k+\lambda_kT_{\xi_k}(x_k)$
\EndFor \\
\Return $x_K$
\end{algorithmic}
\end{algorithm}

As is standard for fixed-point methods, we assume the operator $T_\xi$ to be nonexpansive for $\mathcal D$-almost all $\xi$, and the set of fixed points of $T$, denoted by $\Fix(T)$, to be non-empty.

When dealing with stochastic algorithms, one typically needs to make an assumption about the estimator $T_\xi$. The usual one in the literature \cite{bravo_stochastic_2024,bravo_stochastic_2026,iiduka_minibatch_2026} is to assume that the variance of the stochastic oracle is uniformly bounded. Specifically, it is typically assumed that there exists a $\sigma^2$ such that, for all $x\in \H$, 
\[
    \Exp{\|(T-T_\xi)x\|^2}\le \sigma^2<+\infty.
\]
In the functional case, this has been shown to be unverifiable or unrealistic in practice \cite{bottou_optimization_2018}, which creates a significant gap between theoretical guarantees and practical applicability. The literature therefore has started avoiding making this assumption \cite{bach_nonasymptotic_2011,needell_stochastic_2016,nguyen_sgd_2018,gower_sgd_2019,gower_sgd_2021,khaled_better_2023,cortild_biasoptimal_2025,garrigos_lastiterate_2025}.
This naturally raises the question of whether the above assumption can be relaxed in the fixed point setting while still preserving meaningful guarantees. However, weakening the variance assumptions appears to be nontrivial, as existing analyses heavily rely on uniform variance bounds. 

\textbf{Motivation.} The motivation for this study arose from recent work on differentially private empirical risk minimization \cite{cortild_quadratic_2026}. The setting can naturally be written as the sum of three operators, thereby inviting the natural algorithm of \textit{Stochastic Three-Operator Splitting}. However, existing analyses mostly rely on a uniform variance bound, which did not hold in the studied setting. Instead, the natural assumption was finite variance at a solution, an assumption that has become standard in recent stochastic optimization literature. This observation led us to investigate general stochastic fixed-point iterations under such relaxed noise assumptions. 

\textbf{Contributions.} We provide a convergence analysis of Algorithm \ref{alg:SKM} in a real Hilbert space setting.
Our contributions are as follows.
\begin{enumerate}
    \item We introduce a variance-at-solution assumption, requiring a bounded variance at a single solution $p\in \Fix(T)$, namely that 
    \[
        \Exp{\|(T-T_\xi)p\|^2}\le \sigma_*^2<+\infty.
    \]
    Our key insight is that, by exploiting nonexpansiveness, this local variance assumption propagates to all fixed points and yields a global variance control.
    \item Under this assumption, we establish almost sure convergence of the residuals and weak almost sure convergence of the iterates to a point in $\Fix(T)$. 
    \item We derive convergence rates on the expected random residual and the expected last-iterate residual, and recover the best-known stochastic oracle complexities $\mathcal O(\varepsilon^{-4})$ and $\mathcal O(\varepsilon^{-6})$ respectively, previously established under uniformly bounded variance \cite{bravo_stochastic_2024}. We moreover obtain the same complexity rate almost surely on the running minimum residual, up to logarithmic factors. We moreover derive results for horizon-independent step-sizes, matching the horizon-dependent complexities, up to logarithmic factors.
    \item We apply our results to Stochastic Gradient Descent, recovering known guarantees, and to Stochastic Three-Operator Splitting and Stochastic Backward-Forward Splitting, where we obtain the first results that do not make use of uniformly bounded variance assumptions.
\end{enumerate}
These contributions demonstrate that uniformly bounded variance is unnecessarily strong, and can be relaxed to a weaker, local variance assumption without loss of convergence or complexity guarantees.

\textbf{Comparison with Previous Works.} 
Earlier works studied stochastically perturbed Krasnoselskii-Mann iterations in block-coordinate and additive-error frameworks. Combettes and Pesquet \cite{combettes_stochastic_2015} establish almost sure weak convergence results. In a sequel work \cite{combettes_stochastic_2019}, they obtain mean-square and linear convergence, under additional contractive-type assumptions. More recently, Combettes and Madariaga \cite{combettes_geometric_2026} develop a general geometric framework covering randomly relaxed Krasnoselskii-Mann iterations with stochastic additive errors. The stochastic noise assumptions in these works do not cover the variance-at-solution assumption presented in this work. Moreover, these works do not provide comparable finite-time residual bounds in the general nonexpansive setting.

Our work is therefore most closely related, in terms of residual complexity, to the following works.
\begin{itemize}
    \item The closest prior work is due to Bravo and Cominetti \cite{bravo_stochastic_2024}, who study Stochastic Krasnoselskii-Mann iterations in normed spaces. They consider a framework with additive noise and allow martingale-difference perturbations whose variance may grow in a controlled manner, and obtain convergence of the expected residual and almost sure weak convergence of the iterates. The most directly comparable explicit complexity bounds are obtained under uniformly bounded variance. In that regime, they prove a $\mathcal O(\varepsilon^{-6})$ complexity result for the last-iterate residual, and a $\mathcal O(\varepsilon^{-4})$ complexity result for a random iterate in the Euclidean setting, for horizon-dependent step-sizes. For horizon-independent step-sizes, they obtain a residual convergence rate of $\tilde{\mathcal O}(K^{-1/4})$. We recover their results under weaker assumptions.
    
    \item Bravo and Contreras \cite{bravo_stochastic_2026} study Stochastic Halpern iterations in normed spaces, again under the stronger uniformly bounded variance assumption. Their method uses increasing minibatches and yields an $\mathcal O(\varepsilon^{-5})$ oracle complexity on the last-iterate residual.
    
    \item Finally, we were made aware of the work by Iiduka \cite{iiduka_minibatch_2026} during the writing of this manuscript, and note that our work was developed independently. They study Stochastic Krasnoselskii-Mann iterations, still under uniformly bounded variance, and obtain convergence in $\mathcal O(\varepsilon^{-4})$ stochastic oracle calls. They show convergence of the expected residuals and establish convergence of the iterates in finite dimensions, and do not allow for our weaker assumption.
\end{itemize}

The key distinctions are summarized in Table \ref{tab:comparison}.

\begin{table}[H]
    \centering
    \caption{Comparison of finite-horizon complexities to previous works. For ``Assumption'', $\sigma^2<+\infty$ indicates uniformly bounded variance, whereas $\sigma_*^2<+\infty$ indicates bounded variance at a fixed point. ``Oracle Calls'' denotes total number of calls to the stochastic oracle.}
    \begin{threeparttable}
        \begin{tabular}{|c|c|c|C{1.6cm}|C{2.2cm}|C{1cm}|c|}
            \hline
            \textbf{Reference} & \textbf{Algorithm} & \textbf{Assumption} & \textbf{Batch Size}  &\textbf{Oracle Complexity} & \textbf{A.S. Rate?} & \textbf{Metric} \\ 
            \hline
            \cite{bravo_stochastic_2024} & SKM &  $\sigma^2<+\infty$\tnote{1} & Constant & $\mathcal O(\varepsilon^{-6})$ & No & Last-Iterate\tnote{2} \\
    
            \cite{bravo_stochastic_2024} & SKM & $\sigma^2<+\infty$\tnote{1} & Constant & ${\mathcal O}(\varepsilon^{-4})$ & No  & Random \\ 
            
            \cite{bravo_stochastic_2026} & SHalpern & $\sigma^2<+\infty$ & Increasing & $\tilde {\mathcal O}(\varepsilon^{-5})$ & No  & Last-Iterate \\
    
            \cite{iiduka_minibatch_2026} & SKM & $\sigma^2<+\infty$ & Increasing & ${\mathcal O}(\varepsilon^{-4})$ & No  & Running Min \\
            
            Thm \ref{thm:residual} & SKM & $\sigma_*^2<+\infty$ & Constant & $\mathcal O(\varepsilon^{-4})$ & No  & Random \\ 
            Thm \ref{thm:residual} & SKM & $\sigma_*^2<+\infty$ & Constant & $\tilde o(\varepsilon^{-4})$ & Yes  & Running Min \\ 
            Thm \ref{thm:last} & SKM & $\sigma_*^2<+\infty$ & Constant & $\mathcal O(\varepsilon^{-6})$ & No & Last-Iterate \\ 
            \hline
        \end{tabular}
        \begin{tablenotes}[flushleft]
            \footnotesize
            \item[1] They study Martingale Difference Noise, but only obtain explicit rates in the special case of uniformly bounded variance.
            \item[2] Requires the additional assumption that the range of $T$ is bounded.
        \end{tablenotes}
    \end{threeparttable}
    \label{tab:comparison}
\end{table}

In the quasi-contractive setting, the work by Chen et al.~\cite{chen_finitesample_2020} provides a $\mathcal O(K^{-1/2})$ rate for the distance to the unique fixed-point, under the assumption that the second moment of the noise grows at most affinely with the square norm of the current iterate. Their assumption is much weaker than uniformly bounded variance, and is implied by ours.

We finally note the recent work \cite{combettes_almost_2026}, which discusses last-iterate residual rates with guarantees in an almost-sure setting, albeit under a different variance assumption.

\textbf{Structure.} We present our assumptions, main results and a numerical illustration in Section \ref{sec:results}. In Section \ref{sec:apps}, we apply our results to Stochastic Gradient Descent, Stochastic Three-Operator Splitting, Stochastic Backward-Forward, and Stochastic Lifted Three-Operator Splitting. Auxiliary results are gathered in Appendix \ref{sec:aux}.

\section{Convergence Results}\label{sec:results}

In this section, we present our main convergence results for Algorithm \ref{alg:SKM}. We first introduce our standing assumptions, including our variance-at-solution assumption, and establish key lemmas. We then derive convergence guarantees for the iterates and the expected residual. Finally, we present a simple numerical illustration.

\subsection{Assumptions}\label{sec:assumptions}

We begin by introducing the assumptions used throughout the analysis. The first assumption is on the problem at hand.

\begin{assumption}[Problem Assumption]\label{ass:nonexp}
    We assume the operators $T_\xi$ are nonexpansive for $\mathcal D$-almost all $\xi$. We further assume that $\xi\mapsto T_\xi x$ is measurable and Bochner integrable for each $x\in \H$ and define $T\colon \H\to\H$ by $x\mapsto\ExpD{\xi\sim \mathcal D}{T_\xi x}$. Moreover, we assume $\Fix(T)$ to be nonempty.
\end{assumption}

We note that under the Problem Assumption \ref{ass:nonexp}, it immediately holds that $T$ is nonexpansive. 

The second assumption treats the stochasticity in the problem. Our key contribution is that we do not require a uniformly bounded variance, as done in prior works, but merely a bounded variance-at-solution assumption.

\begin{assumption}[Variance Assumption]\label{ass:bounded_some_FP}
    We assume that, for some $p\in \Fix(T)$,
    \[
        \Exp{\|(T_\xi-T)p\|^2}\le \sigma_*^2<+\infty.
    \]
\end{assumption}

The next lemma shows that if the quantity in the Variance Assumption \ref{ass:bounded_some_FP} is bounded for some $p\in \Fix(T)$, then the quantity is constant over all $\Fix(T)$, and hence bounded for all $p\in \Fix(T)$. This parallels the corresponding result in the function setting explored in \cite[Lemma 4.17]{garrigos_handbook_2024}.

\begin{lemma}[Bounded Variance at All Fixed Points]\label{lem:bounded_all_FP}
    Under the Problem Assumption \ref{ass:nonexp} and the Variance Assumption \ref{ass:bounded_some_FP}, it holds that, for all $p\in\Fix(T)$,
    \[
        \Exp{\|(T_\xi-T)p\|^2}\le \sigma_*^2.
    \]
\end{lemma}
\begin{proof}
    Consider any two fixed points $p, p'\in \Fix(T)$. As such, by convexity of $x\mapsto \|x\|^2$, it holds that
    \[
        \|p-p'\|^2=\|\Exp{T_\xi p}-\Exp{T_\xi p'}\|^2\le \Exp{\|T_{\xi}p-T_{\xi}p'\|^2}\le \|p-p'\|^2.
    \]
    Therefore equality holds, namely $\|\Exp{T_\xi p-T_\xi p'}\|^2=\Exp{\|T_\xi p-T_\xi p'\|^2}$, and thus $T_\xi p-T_\xi p'$ has zero variance and is almost surely constant. In particular, it is equal to its expectation, namely $T_\xi p-T_\xi p'=Tp-Tp'$ almost surely, which implies that $(T_\xi-T)p=(T_\xi-T)p'$ for all $p, p'\in \Fix(T)$ almost surely, and $p\mapsto \Exp{\|(T_\xi-T)p\|^2}$ is constant over $\Fix(T)$, as claimed.
\end{proof}

Moreover, if the quantity in Assumption \ref{ass:bounded_some_FP} is bounded for some $p\in \Fix(T)$, then we may bound it for all $x\in \mathcal H$, albeit with a non-constant right-hand side. We refer to this as a variance transfer result. The below generalizes the functional setting \cite[Lemmas 4.19 and 4.20]{garrigos_handbook_2024}.

\begin{lemma}[Variance Transfer Lemma]\label{lem:transfer}
    Under the Problem Assumption \ref{ass:nonexp} and the Variance Assumption \ref{ass:bounded_some_FP}, it holds that, for all $x\in \H$ and $p\in \Fix(T)$,
    \[
        \Exp{\|(T_\xi-T)x\|^2}\le 2\sigma_*^2+2\|x-p\|^2.
    \]
\end{lemma}
\begin{proof}
    We have, by the triangle inequality and the Variance Assumption \ref{ass:bounded_some_FP},
    \begin{align*}
        \sqrt{\Exp{\|(T_\xi-T)x\|^2}}
        &\le \sqrt{\Exp{\|(T_\xi-T)p\|^2}}+\sqrt{\Exp{\|T_\xi x-T_\xi p-(Tx-Tp)\|^2}} \\
        &\le \sigma_*+\sqrt{\Exp{\|T_\xi x-T_\xi p-(Tx-Tp)\|^2}} \\
        & = \sigma_* + \sqrt{\Exp{\|T_\xi x-T_\xi p\|^2-2\langle T_\xi x-T_\xi p,Tx-Tp\rangle+\|Tx-Tp\|^2}}\\
        & = \sigma_* + \sqrt{\Exp{\|T_\xi x-T_\xi p\|^2}-\|Tx-Tp\|^2} \\
        & \le \sigma_* + \|x-p\|,
    \end{align*}
    where we discarded the negative term and used the nonexpansiveness of $T_\xi$ in the last inequality. The result follows by the standard inequality $(a+b)^2\le 2a^2+2b^2$. We note this result holds for all $p\in \Fix(T)$, and not only the one in the Variance Assumption \ref{ass:bounded_some_FP}, by Lemma \ref{lem:bounded_all_FP}.
\end{proof}

\begin{remark} \label{rem:Young}
    In Lemma \ref{lem:transfer}, one can apply Young's inequality more generally to deduce that 
    \begin{align*}
        \Exp{\|(T_\xi-T)x\|^2}
        &\le (1+\eta^{-1})\sigma_*^2+(1+\eta)\|x-p\|^2,
    \end{align*}
    for any $\eta>0$. In the case of interpolation, where $\sigma_*^2=0$, the first term disappears, and we therefore deduce that 
    $$\Exp{\|(T_\xi-T)x\|^2} \le \|x-p\|^2.$$
    This results in a reduction of the bias terms in Theorems \ref{thm:residual} and \ref{thm:last}. Although we are not aiming at optimizing the constants in these expressions, this fact is interesting for the numerical comparison in Section \ref{ssec:numerical}.  
\end{remark}

\subsection{Convergence Results}

We let $\mathcal F_k=\sigma(x_0, \xi_0, \xi_1, \ldots, \xi_{k-1})$ be a filtration adapted to the iterate $x_k$. In particular, $\xi_k$ is independent of $\mathcal F_k$. For simplicity of notation, we write $\E_k$ as $\Exp{\cdot|\mathcal F_k}$ and $\E$ as a total expectation.

To establish convergence, we impose assumptions on the step-size sequence. We denote by $\ell^p(\N)$ all sequences $(\lambda_k)_{k\in \N}\subset \R$ such that $\sum_{k=0}^\infty |\lambda_k|^p<+\infty$, where $p\geq 1$.
\begin{assumption}[Parameter Assumption]\label{ass:stepsize}
    The relaxation parameters $(\lambda_k)\subset (0,1)$ satisfy $(\lambda_k)\in \ell^2(\N)\backslash \ell^1(\N)$. 
\end{assumption}

We first analyze the setting where the realizations $T_\xi$ are nonexpansive, namely that $\|T_\xi x-T_\xi y\|\le \|x-y\|$ for all $x, y\in \H$, and later consider the setting where, additionally, $T$ is $q$-quasi-contractive for $q\in (0, 1)$, namely that $\|Tx-p\|\le q\|x-p\|$ for all $x\in H$ and $p\in \Fix(T)$.

\subsubsection{Nonexpansive Setting}

We define the following quantity for notational convenience
\[
    \Lambda_k=\prod_{i=0}^{k-1}(1+2\lambda_i^2)^{-1}.
\]
We note that $\Lambda_K\ge \bar\Lambda$ for some $\bar\Lambda>0$ under the assumption that $(\lambda_k)\in \ell^2(\N)$. In fact, since $\ln(1+x)\le x$ for all $x>-1$,
\begin{equation}\label{eq:bound_Lambda}
    \Lambda_K^{-1}=\prod_{k=0}^{K-1}(1+2\lambda_k^2)=\exp\left(\sum_{k=0}^{K-1}\ln(1+2\lambda_k^2)\right)\le \exp\left(2\sum_{k=0}^{K-1} \lambda_k^2\right),
\end{equation}
so we may pick $\bar\Lambda^{-1}=\exp\left(2\sum_{k=0}^{\infty} \lambda_k^2\right)<+\infty$.

We also introduce a random index $N_K$, independent of the algorithmic randomness, defined by
\begin{equation}\label{eq:def_prob}
    \P(N_K=k)=\frac{\Lambda_{k+1}\lambda_k(1-\lambda_k)}{\sum_{k'=0}^{K-1}\Lambda_{k'+1}\lambda_{k'}(1-\lambda_{k'})}\quad \text{for $k=0, \ldots, K-1$}.
\end{equation}
This index allows us to convert weighted averages of the residual into guarantees on a single (random) iterate.

\begin{remark}\label{rem:decreasing_relaxation_parameters}
    The Parameter Assumption \ref{ass:stepsize} is naturally satisfied for step-sizes $\lambda_k = \lambda_0 (k+1)^{-a}$ for $a\in (1/2, 1]$ and $\lambda_0\in (0, 1)$, since
    \[
        \sum_{k=0}^{K-1}(k+1)^{-\alpha}=\begin{cases}
            \Theta(K^{1-\alpha})\quad &\text{if $\alpha<1$}, \\
            \Theta(\ln(K))\quad &\text{if $\alpha=1$}, \\
            \Theta(1)\quad &\text{if $\alpha>1$}.
        \end{cases}
    \]
    It also holds for $\lambda_k=\lambda_0(k+1)^{-a}\ln(k+e)^{-1}$ with $\lambda_0\in (0, 1)$ for $a\in [1/2, 1]$, since
    \[
        \sum_{k=0}^{K-1}(k+1)^{-a}\ln(k+e)^{-1}=\begin{cases}
            \Theta(K^{1-a}\ln(K)^{-1})\quad &\text{if $a<1$}, \\
            \Theta(\ln(\ln(K)))&\text{if $a=1$},
        \end{cases}
    \]
    and, for all $a\in [1/2, 1]$,
    \[
        \sum_{k=0}^{K-1}(k+1)^{-2a}\ln(k+e)^{-2}=\Theta(1).
    \]
\end{remark}

\begin{lemma}[Decrease Inequality]\label{lem:ineq}
    Assume the Problem Assumption \ref{ass:nonexp} and the Variance Assumption \ref{ass:bounded_some_FP} hold, and let $(x_k)$ be the random iterates determined by Algorithm \ref{alg:SKM}. Then it holds that, for every $p\in \Fix(T)$,
    \[
        \ExpD{k}{\|x_{k+1}-p\|^2}\le (1+2\lambda_k^2)\|x_k-p\|^2 -\lambda_k(1-\lambda_k)\|Tx_k-x_k\|^2 + 2\lambda_k^2\sigma_*^2, 
    \]
\end{lemma}
\begin{proof}
    It holds that
    \begin{align*}
        \ExpD{k}{\|x_{k+1}-p\|^2}
        &=\ExpD{k}{\|(1-\lambda_k)(x_k-p)+\lambda_k(T_{\xi_k}x_k-p)\|^2} \\
        &=(1-\lambda_k)^2\|x_k-p\|^2+\lambda_k^2\ExpD{k}{\|T_{\xi_k}x_k-p\|^2}+2\lambda_k(1-\lambda_k)\ExpD{k}{\langle x_k-p, T_{\xi_k}x_k-p\rangle} \\
        &= (1-\lambda_k)^2\|x_k-p\|^2+\lambda_k^2\ExpD{k}{\|T_{\xi_k}x_k-p\|^2}+2\lambda_k(1-\lambda_k)\langle x_k-p, Tx_k-p\rangle.
    \end{align*}
    By writing 
    \[
        2\langle x_k-p, Tx_k-p\rangle=\|x_k-p\|^2+\|Tx_k-p\|^2-\|Tx_k-x_k\|^2
    \]
    and 
    \begin{align*}
        \ExpD{k}{\|T_{\xi_k}x_k-p\|^2}
        &=\ExpD{k}{\|T_{\xi_k}x_k-x_k\|^2}+\|x_k-p\|^2+2\ExpD{k}{\langle T_{\xi_k}x_k-x_k, x_k-p\rangle} \\
        &=\ExpD{k}{\|T_{\xi_k}x_k-x_k\|^2}+\|x_k-p\|^2+2\langle Tx_k-x_k, x_k-p\rangle \\
        &=\ExpD{k}{\|T_{\xi_k}x_k-x_k\|^2}-\|Tx_k-x_k\|^2+\|Tx_k-p\|^2,
    \end{align*}
    we obtain that
    \begin{align*}
        \ExpD{k}{\|x_{k+1}-p\|^2}
        &=(1-\lambda_k)\|x_k-p\|^2+\lambda_k^2\ExpD{k}{\|T_{\xi_k}x_k-x_k\|^2}+\lambda_k\|Tx_k-p\|^2-\lambda_k \|Tx_k-x_k\|^2 \\
        &\le \|x_k-p\|^2+\lambda_k^2\ExpD{k}{\|T_{\xi_k}x_k-x_k\|^2}-\lambda_k^2 \|Tx_k-x_k\|^2-\lambda_k(1-\lambda_k)\|Tx_k-x_k\|^2 \\
        &\le (1+2\lambda_k^2)\|x_k-p\|^2 -\lambda_k(1-\lambda_k)\|Tx_k-x_k\|^2 + 2\lambda_k^2\sigma_*^2, 
    \end{align*}
    where the first inequality follows by nonexpansiveness of $T$, and the second by the Variance Transfer Lemma \ref{lem:transfer}. 
\end{proof}

Items 2 and 3 of the following lemma are standard consequences of stochastic quasi-Fejér and almost-supermartingale arguments, see for instance \cite[Proposition 2.3]{combettes_stochastic_2015}. We provide the short derivation below in order to identify explicitly the problem-dependent terms and establish the bound in Item 1.

\begin{lemma}\label{lem:bounded_iterates}\label{lem:bounded_its}
    Assume the Problem Assumption \ref{ass:nonexp} and the Variance Assumption \ref{ass:bounded_some_FP} hold, and let $(x_k)$ be the random iterates determined by Algorithm \ref{alg:SKM}. Under the Parameter Assumption \ref{ass:stepsize}, the following hold
    \begin{enumerate}
        \item $\Exp{\|x_k-p\|^2}$ is bounded uniformly in $k$ for all $p\in \Fix(T)$,
        \item $\|x_k-p\|$ is almost surely convergent for all $p\in \Fix(T)$, and
        \item $\sum_{k=0}^\infty \lambda_k(1-\lambda_k)\|Tx_k-x_k\|^2<+\infty$ almost surely.
    \end{enumerate}
\end{lemma}
\begin{proof}
\begin{enumerate}
    \item By multiplying Lemma \ref{lem:ineq} by $\Lambda_{k+1}$ on both sides, noting that $\Lambda_{k+1}(1+2\lambda_k^2)=\Lambda_k$ and that $\Lambda_{k+1}\le 1$, and taking total expectation, we obtain
    \begin{equation}\label{eq:to_sum}
        \Lambda_{k+1}\Exp{\|x_{k+1}-p\|^2}\le \Lambda_{k}\Exp{\|x_k-p\|^2}-\Lambda_{k+1}\lambda_k(1-\lambda_k)\Exp{\|Tx_k-x_k\|^2}+2\lambda_k^2\sigma_*^2.
    \end{equation}
    By summing for $k=0, \ldots, K-1$ and bounding away the negative term, we obtain
    \[
    	\Exp{\|x_{K}-p\|^2}\le \frac{\|x_0-p\|^2+2\sum_{k=0}^{K-1}\lambda_k^2\sigma_*^2}{\Lambda_{K}}.
    \]
    In particular, since $\sum_{k=0}^{K-1} \lambda_k^2\le \sum_{k=0}^\infty \lambda_k^2<+\infty$ and $\Lambda_K\ge \bar\Lambda>0$, we get that $\Exp{\|x_K-p\|^2}$ is uniformly bounded in $K$.
    \item[2+3.] The conditions of the Robbins-Siegmund Almost-Supermartingale Theorem \ref{thm:robbins_siegmund} with
    \[
        Y_k=\|x_k-p\|^2, \quad X_k=\lambda_k(1-\lambda_k)\|Tx_k-x_k\|^2, \quad Z_k=2\lambda_k^2\sigma_*^2, \quad \text{and}\quad \eta_k=2\lambda_k^2
    \]
    hold by Lemma \ref{lem:ineq} and the Parameter Assumption \ref{ass:stepsize}. As such, we get that $\|x_k-p\|^2$ converges almost surely and that $\sum_{k=0}^\infty \lambda_k(1-\lambda_k)\|Tx_k-x_k\|^2<+\infty$ almost surely. \qedhere
\end{enumerate}
\end{proof}

The following theorem provides the main convergence rate theorem, in which we show that the expectation of the residual at a random iterate converges. This immediately implies convergence of the minimal iterate along the trajectory. However, in practice, that is inconvenient as it requires evaluating all the residuals, which is precisely what the stochastic setting seeks to avoid. The idea of introducing the random iterate is inspired by \cite{bravo_stochastic_2024}. 

\begin{theorem}[Bounds on Random and Minimal Running Residuals]\label{thm:residual}
    Assume the Problem Assumption \ref{ass:nonexp} and the Variance Assumption \ref{ass:bounded_some_FP} hold, and let $(x_k)$ be the random iterates determined by Algorithm \ref{alg:SKM}. 
    Then, for $K\ge 1$,
    \[
        \Exp{\|Tx_{N_K}-x_{N_K}\|^2}\le \frac{\dist(x_0, \Fix(T))^2+2\sum_{k=0}^{K-1}\lambda_k^2\sigma_*^2}{\Lambda_{K}\sum_{k=0}^{K-1}\lambda_k(1-\lambda_k)},
    \]
    where $N_K$ is sampled according to Equation \eqref{eq:def_prob}, and $\E$ is taken over all sources of randomness. Moreover, under the Parameter Assumption \ref{ass:stepsize}, for $K\ge 1$,
    \[
        \Exp{\|Tx_{N_K}-x_{N_K}\|}\le \mathcal O\left(\frac{1}{\sqrt{\sum_{k=0}^{K-1}\lambda_k(1-\lambda_k)}}\right) \quad \text{and}\quad \min_{k\le K-1}\|Tx_{k}-x_{k}\|\le o\left(\frac{1}{\sqrt{\sum_{k=0}^{K-1}\lambda_k(1-\lambda_k)}}\right),
    \]
    where the latter holds almost surely.
\end{theorem}
\begin{proof}
    By summing Equation \eqref{eq:to_sum} for $k=0, \ldots, K-1$, we obtain
    \[
        \Lambda_{K}\Exp{\|x_{K}-p\|^2}\le \|x_0-p\|^2-\sum_{k=0}^{K-1}\Lambda_{k+1}\lambda_k(1-\lambda_k)\Exp{\|Tx_k-x_k\|^2}+2\sum_{k=0}^{K-1}\lambda_k^2\sigma_*^2.
    \]
    Bounding $\Lambda_{K}\Exp{\|x_{K}-p\|^2}$ by $0$ and rearranging yields
    \[
        \sum_{k=0}^{K-1}\Lambda_{k+1}\lambda_k(1-\lambda_k)\Exp{\|Tx_k-x_k\|^2}\le \|x_0-p\|^2+2\sum_{k=0}^{K-1}\lambda_k^2\sigma_*^2,
    \]
    and specifically, we obtain
    \begin{equation*}\label{eq:intermediate}
        \frac{\sum_{k=0}^{K-1}\Lambda_{k+1}\lambda_k(1-\lambda_k)\Exp{\|Tx_k-x_k\|^2}}{\sum_{k=0}^{K-1}\Lambda_{k+1}\lambda_k(1-\lambda_k)}\le \frac{\|x_0-p\|^2+2\sum_{k=0}^{K-1}\lambda_k^2\sigma_*^2}{\sum_{k=0}^{K-1}\Lambda_{k+1}\lambda_k(1-\lambda_k)}.
    \end{equation*}
    The left-hand side coincides with $\Exp{\|Tx_{N_K}-x_{N_K}\|^2}$, and the right-hand side yields the wanted inequality by noting that $(\Lambda_k)$ is nonincreasing and by taking the infimum over all $p\in \Fix(T)$. The conclusion follows as $\Exp{\|Tx_{N_K}-x_{N_K}\|}^2\le \Exp{\|Tx_{N_K}-x_{N_K}\|^2}$. The final almost sure result follows from Lemma \ref{lem:bounded_its} and Lemma \ref{lem:small_o}.
\end{proof}

The previous proof technique naturally yields bounds on the expected residual. In order to get last-iterate bounds, we define a ghost-sequence of points not visited by the algorithm but only introduced for theoretical purposes, as done in \cite{du_linear_2019}, for instance. 
\begin{theorem}[Bound on Last-Iterate Residual]\label{thm:last}
	Assume the Problem Assumption \ref{ass:nonexp} and the Variance Assumption \ref{ass:bounded_some_FP} hold, and let $(x_k)$ be the random iterates determined by Algorithm \ref{alg:SKM}. 
    Then, for $K\ge 1$,
    \[
        \Exp{\|Tx_K-x_K\|^2}\le \left(16Q_K+\frac{2}{\sum_{k=0}^{K-1}\lambda_k(1-\lambda_k)}\right)\frac{\dist(x_0, \Fix(T))^2+2\sum_{k=0}^{K-1}\lambda_k^2\sigma_*^2}{\Lambda_K}+16Q_K \sigma_*^2,
    \]
    where $Q_K=\frac{\sum_{k=0}^{K-1}\Lambda_{k+1}\lambda_k(1-\lambda_k)\sum_{i=k}^{K-1}\lambda_i^2}{{\sum_{k=0}^{K-1}\Lambda_{k+1}\lambda_k(1-\lambda_k)}}$.
\end{theorem}
\begin{proof}
	Fix any $p\in \Fix(T)$. By the bound obtained in the proof of Lemma \ref{lem:bounded_its}, we know that $\Exp{\|x_k-p\|^2}$ for $k=0, \ldots, K-1$ is bounded by a constant $M<+\infty$ given by
	\[
		M=\frac{\|x_0-p\|^2+2\sum_{k=0}^{K-1}\lambda_k^2\sigma_*^2}{\Lambda_K}
	\]
	We start by fixing an index $k\in \{0, \ldots, K-1\}$, and define the sequence $\{y_i^{(k)}\}_{i\ge k}$, given by 
	\[
		y_k^{(k)} = x_k, \quad y_{i+1}^{(k)}=(1-\lambda_{i})y_i^{(k)}+\lambda_i Ty_i^{(k)}\quad \text{for $i\ge k$}.
	\]
	Specifically, the sequence $\{y_i^{(k)}\}_{i\ge k}$ starts at $x_k$, and follows the deterministic Krasnoselskii-Mann iterations $x^+=S_\lambda x$, for $S_\lambda\colon x\mapsto (1-\lambda)x+\lambda Tx$. As such, since the deterministic Krasnoselskii-Mann mapping is nonexpansive and $\ExpD{i}{(T-T_{\xi_i})x_i}=0$, it holds that, for $k\le i\le K-1$,
	\begin{align*}
		\ExpD{i}{\|x_{i+1}-y_{i+1}^{(k)}\|^2}
		&=\ExpD{i}{\|S_{\lambda_i}x_i-S_{\lambda_i}y_i^{(k)}-S_{\lambda_i}x_i+((1-\lambda_i)x_i+\lambda_iT_{\xi_i}x_i\|^2}\\
		&=\ExpD{i}{\|S_{\lambda_i}x_i-S_{\lambda_i}y_i^{(k)}+\lambda_i(T_{\xi_i}-T)x_i\|^2} \\
		&\le\ExpD{i}{\|x_i-y_i^{(k)}\|^2}+\lambda_i^2\ExpD{i}{\|(T_{\xi_i}-T)x_i\|^2} \\
		&\le\ExpD{i}{\|x_i-y_i^{(k)}\|^2}+2\lambda_i^2(\|x_i-p\|^2+\sigma_*^2),
	\end{align*}
    which, taking total expectation, yields
    \[
        \Exp{\|x_{i+1}-y_{i+1}^{(k)}\|^2}\le \Exp{\|x_i-y_i^{(k)}\|^2}+2\lambda_i^2(M+\sigma_*^2).
    \]
	Iterating this yields that
	\[
		\Exp{\|x_K-y_K^{(k)}\|^2}\le \Exp{\|x_k-y_k^{(k)}\|^2}+\sum_{i=k}^{K-1}\lambda_i^2(2M+2\sigma_*^2).
	\]
	As such, we obtain that
	\begin{align*}
		\|Tx_K-x_K\|
		&=\|Tx_K-Ty_K^{(k)}+Ty_K^{(k)}-y_K^{(k)}+y_K^{(k)}-x_K\| \\
		&\le 2\|x_K-y_K^{(k)}\| + \|Ty_K^{(k)}-y_K^{(k)}\| \\
		&\le 2\|x_K-y_K^{(k)}\|+ \|Tx_k-x_k\|,
	\end{align*}
	where the latter follows by the monotonic nonincrease of deterministic Krasnoselskii-Mann iterations. Specifically, 
	\[
		\Exp{\|Tx_K-x_K\|^2}\le 16\sum_{i=k}^{K-1}\lambda_i^2(M+\sigma_*^2)+ 2\Exp{\|Tx_k-x_k\|^2},
	\]
	which, after multiplication by $\Lambda_{k+1}\lambda_k(1-\lambda_k)$ and summing over $k=0, \ldots, K-1$, we get
	\[
		\sum_{k=0}^{K-1}\Lambda_{k+1}\lambda_k(1-\lambda_k)\Exp{\|Tx_K-x_K\|^2}\le 16\sum_{k=0}^{K-1}\Lambda_{k+1}\lambda_k(1-\lambda_k)\sum_{i=k}^{K-1}\lambda_i^2(M+\sigma_*^2)+ 2\sum_{k=0}^{K-1}\Lambda_{k+1}\lambda_k(1-\lambda_k)\Exp{\|Tx_k-x_k\|^2}.
	\]
	Dividing by $\sum_{k=0}^{K-1}\Lambda_{k+1}\lambda_k(1-\lambda_k)$ and using the bound from Theorem \ref{thm:residual} then yields
	\[
		\Exp{\|Tx_K-x_K\|^2}\le 16\frac{\sum_{k=0}^{K-1}\Lambda_{k+1}\lambda_k(1-\lambda_k)\sum_{i=k}^{K-1}\lambda_i^2}{{\sum_{k=0}^{K-1}\Lambda_{k+1}\lambda_k(1-\lambda_k)}}(M+\sigma_*^2)+2\frac{\|x_0-p\|^2+2\sum_{k=0}^{K-1}\lambda_k^2\sigma_*^2}{\sum_{k=0}^{K-1}\Lambda_{k+1}\lambda_k(1-\lambda_k)},
	\]
	which gives the wanted inequality by replacing $M$ by its value, bounding $\Lambda_{k+1}\ge \Lambda_K$ for all $k=0, \ldots, K-1$, and taking the infimum over all $p\in \Fix(T)$.
\end{proof}

\begin{corollary}[Convergence Rates of Expected Residual]\label{coro:complexity}
    Assume the Problem Assumption \ref{ass:nonexp} and the Variance Assumption \ref{ass:bounded_some_FP} hold, and let $(x_k)$ be the random iterates determined by Algorithm \ref{alg:SKM}. Then it holds that
    \begin{enumerate}
        \item (Polynomial Decreasing Relaxation Parameter) If $\lambda_k=\lambda_0 \cdot (k+1)^{-a}$ for $a\in (1/2, 1]$, then, for $K\ge 1$,
        \[
            \Exp{\|T x_{N_K}- x_{N_K}\|}\le 
            \begin{dcases}
                \mathcal O\left(K^{-\frac{1-a}2}\right)\quad &\text{if $a<1$} \\
                \mathcal O\left(\ln(K)^{-\frac12}\right) & \text{if $a=1$}
            \end{dcases} \quad \text{and}\quad \min_{k\le K-1}\|Tx_{k}-x_{k}\|\le 
            \begin{dcases}
                o\left(K^{-\frac{1-a}2}\right)\quad &\text{if $a<1$} \\
                o\left(\ln(K)^{-\frac12}\right) & \text{if $a=1$},
            \end{dcases}
        \]
        where the latter holds almost surely as $K\to\infty$. Moreover, 
        \[
            \Exp{\|Tx_K-x_K\|}\le \begin{cases}
                \mathcal O(K^{-\frac{2a-1}{2}})\quad &\text{if $1/2<a<2/3$}, \\
                \mathcal O(\ln(K)^{\frac12}K^{-\frac{1}{6}})\quad &\text{if $a=2/3$}, \\
                \mathcal O(K^{-\frac{1-a}{2}})\quad &\text{if $2/3<a<1$}, \\
                \mathcal O(\ln(K)^{-\frac12})\quad &\text{if $a=1$}.
            \end{cases}
        \]
        \item (Poly-Logarithmic Decreasing Relaxation Parameter) If $\lambda_k=\lambda_0 \cdot (k+1)^{-a}\ln(k+e)^{-1}$ for $a\in [1/2, 1]$, then, for $K\ge 3$,
        \[
            \Exp{\|T x_{N_K}- x_{N_K}\|}\le \begin{cases}
                \mathcal O(K^{-\frac{1-a}{2}}\ln(K)^{\frac12})\quad &\text{if $1/2\le a<1$}, \\
                \mathcal O((\ln\ln K)^{-\frac12}) & \text{if $a=1$},
            \end{cases}
        \]
        and, as $K\to \infty$,
        \[
            \min_{k\le K-1}\|Tx_{k}-x_{k}\|\le \begin{cases}
                o(K^{-\frac{1-a}{2}}\ln(K)^{\frac12})\ &\text{if $1/2\le a<1$}, \\
                o((\ln\ln K)^{-\frac12}) & \text{if $a=1$},
            \end{cases}
        \]
        where the latter holds almost surely.
        \item (Constant Horizon-Dependent Relaxation Parameter) If $\lambda_k\equiv \lambda_0 K^{-a}$ with $a\in [1/2, 1)$ for $K\ge 1$, then 
        \[
            \Exp{\|Tx_{N_K}-x_{N_K}\|}\le \mathcal O\left(K^{-\min(a, 1-a)/2}\right)\quad \text{and}\quad \Exp{\|Tx_K-x_K\|}\le \mathcal O\left(K^{-\min(-1+2a, 1-a)/2}\right).
        \]
    \end{enumerate}
\end{corollary}
\begin{proof} The result follows by Theorems \ref{thm:residual} and \ref{thm:last}, by substituting appropriate bounds for $\sum_{k=0}^{K-1}\lambda_k$, $\sum_{k=0}^{K-1}\lambda_k^2$ and $\Lambda_K$.
\begin{enumerate} 
    \item The bounds for the first two inequalities follow directly from Remark \ref{rem:decreasing_relaxation_parameters} and Theorem \ref{thm:residual}. For the last inequality, we note that 
    \[
        \sum_{i=k}^{K-1}\lambda_i^2= \lambda_0^2\sum_{i=k}^{\infty}(i+1)^{-2a}\le \mathcal O((k+1)^{1-2a}),
    \]
    and since $1-\lambda_k\ge 1/2$ eventually, we get that 
    \[
        Q_K\le \mathcal O\left(\frac{\sum_{k=0}^{K-1}(k+1)^{1-3a}}{\sum_{k=0}^{K-1}(k+1)^{-a}}\right),
    \]
    and the result again follows from Remark \ref{rem:decreasing_relaxation_parameters} and Theorem \ref{thm:last}.
    \item This follows directly from Remark \ref{rem:decreasing_relaxation_parameters} and Theorem \ref{thm:residual}.
    \item It holds that 
    \[
        \sum_{k=0}^{K-1}\lambda_k^2= \lambda_0^2K^{1-2a}\quad\text{and}\quad \sum_{k=0}^{K-1}\lambda_k(1-\lambda_k)= \lambda_0K^{1-a}(1-\lambda_0K^{-a})=\lambda_0K^{1-a}-\lambda_0^2K^{1-2a}.
    \]
    Moreover, $\Lambda_K^{-1}\le e^{2\lambda_0^2K^{1-2a}}\le e^{2\lambda_0^2}$ by Equation \eqref{eq:bound_Lambda}, and the first conclusion follows by Theorem \ref{thm:residual}. Moreover, $Q_K\le \lambda_0^2K^{1-2a}$, and the second conclusion follows by Theorem \ref{thm:last}.
    \qedhere
\end{enumerate}
\end{proof}

\begin{theorem}[Strong Convergence of Residuals and Weak Convergence of Iterates]\label{thm:weak_conv}
    Assume the Problem Assumption \ref{ass:nonexp} and the Variance Assumption \ref{ass:bounded_some_FP} hold, and let $(x_k)$ be the random iterates determined by Algorithm \ref{alg:SKM}. Moreover, assume the Parameter Assumption \ref{ass:stepsize} holds. Then, almost surely, it holds that 
    \begin{enumerate}
        \item $\|Tx_k-x_k\|$ converges to $0$, and
        \item $x_k$ converges weakly to an element in $\Fix(T)$.
    \end{enumerate}
\end{theorem}
\begin{proof}
We prove the two statements separately.
\begin{enumerate}
    \item The conclusion follows from Lemma \ref{lem:to_zero} with $\eta_k=\lambda_k$, $r_k=\|Tx_k-x_k\|$, $w_k=(T_{\xi_k}-T)x_k$. We shall show that the conditions of Lemma \ref{lem:to_zero} are almost surely satisfied:
    \begin{enumerate}
        \item The condition $\sum_{k=0}^\infty \eta_kr_k^2<+\infty$ is almost surely satisfied. Indeed, note that by the Parameter Assumption \ref{ass:stepsize}, $\sum_{k=0}^\infty{\lambda_k^2}<+\infty$, and hence $\lambda_k^2\to 0$. As such, there exists a $k_0\ge 0$ such that, for all $k\ge k_0$, $\lambda_k\le \frac{1}{2}$, in which case $\lambda_k(1-\lambda_k)\ge \frac{\lambda_k}{2}$. As such, almost surely,
        \[
            \sum_{k=k_0}^\infty \eta_kr_k^2=\sum_{k=k_0}^\infty \lambda_k\|Tx_k-x_k\|^2\le 2\sum_{k=k_0}^\infty \lambda_k(1-\lambda_k)\|Tx_k-x_k\|^2<+\infty,
        \]
        where the final inequality follows from Lemma \ref{lem:bounded_iterates}.
        \item The main inequality condition is satisfied for $L=2$. In fact, consider
        \begin{align*}
            |r_{k+\tau}-r_k|
            &= \left|\|Tx_{k+\tau}-x_{k+\tau}\|-\|Tx_k-x_k\|\right| 
            \le \left\|Tx_{k+\tau}-x_{k+\tau}-Tx_k+x_k\right\| 
            \le 2\|x_{k+\tau}-x_k\|\\
            &=2\left\|\sum_{i=k}^{k+\tau-1}x_{i+1}-x_{i}\right\|
            =2\left\|\sum_{i=k}^{k+\tau-1}\lambda_{i}(Tx_{i}-x_{i})+\lambda_{i}(T_{\xi_{i}}-T)x_{i}\right\| \\
            &\le 2\sum_{i=k}^{k+\tau-1}\lambda_{i} r_{i}+2\left\|\sum_{i=k}^{k+\tau-1}\lambda_{i}(T_{\xi_{i}}-T)x_{i}\right\| 
            = 2\left(\sum_{i=k}^{k+\tau-1}\eta_{i} r_{i}+\left\|\sum_{i=k}^{k+\tau-1}\eta_{i}w_{i}\right\|\right),
        \end{align*}
        where the inequalities follow by the triangle inequality and the nonexpansiveness of $T$.
        \item The quantity $(\sum_{k=0}^K \eta_kw_k)_K$ converges almost surely and in $L^2$. Indeed, since $(w_k)$ is a martingale difference sequence with respect to $(\mathcal F_{K+1})_K$ and by applying Lemma \ref{lem:transfer}, we get, for any $K\ge 0$,
        \[
            \Exp{\left\|\sum_{k=0}^{K}\eta_{k}w_{k}\right\|^2}=\sum_{k=0}^{K}\lambda_k^2\Exp{\left\|(T_{\xi_k}-T)x_k\right\|^2}\le 2\sigma_*^2\sum_{k=0}^{K}\lambda_k^2+2\sum_{k=0}^{K}\lambda_k^2\Exp{\|x_k-p\|^2},
        \]
        which is bounded since $(\lambda_k)\in \ell^2(\N)$ and $\Exp{\|x_k-p\|^2}$ is bounded by Lemma \ref{lem:bounded_its}. As such, the martingale sequence $(\sum_{k=0}^{K}\eta_{k}w_{k})_{K\ge 0}$ is uniformly $L^2$-bounded, so it converges almost surely and in $L^2$.
    \end{enumerate}
    Specifically, the conditions of Lemma \ref{lem:to_zero} are almost surely satisfied, and hence $r_k=\|Tx_k-x_k\|\to 0$ almost surely.
    \item Lemma \ref{lem:bounded_its} proves that $(x_k)$ is almost surely bounded.
    Consider any weak limit point $x_*$ of $(x_k)$. By Part 1, it holds that $Tx_k-x_k\to 0$ almost surely, and by Browder's Demiclosedness Principle \cite{browder_nonexpansive_1965}, since $T$ is nonexpansive, it holds that $x_*\in \Fix(T)$ almost surely. Moreover, for any $p\in \Fix(T)$, it holds that $\|x_k-p\|$ converges almost surely by Lemma \ref{lem:bounded_iterates}. Therefore there exists $\Omega_p$, dependent on $p$, representing the realization of the algorithmic randomness, such that $\|x_k-p\|$ converges on $\Omega_p$ and such that $\P(\Omega_p)=1$. We take a countable dense subset $\{p_j\}$ of $\Fix(T)$, which exists as $\H$ is separable, and denote $\Omega=\cap_j \Omega_{p_j}$, such that $\P(\Omega)=1$. Then, for any $q\in \Fix(T)$, we construct a sequence $(q_j)\subset \{p_j\}$ such that $q_j\to q$. It holds that 
    \[
        \left|\|x_k-q\|-\|x_k-q_j\|\right|\le \|q-q_j\|\stackrel{j\to \infty}{\rightarrow} 0.
    \]
    Since $\|x_k-q_j\|$ converges on $\Omega$ for all $j\ge0$, $\|x_k-q\|$ must converge on $\Omega$. As such, $\|x_k-p\|$ converges almost surely simultaneously for all $p\in \Fix(T)$ on the common event $\Omega$.
    As such, by Opial's Lemma \cite{opial_weak_1967}, $(x_k)$ must converge weakly to a point in $\Fix(T)$ almost surely. \qedhere
\end{enumerate}
\end{proof}

\subsubsection{Quasi-Contractive Setting}

We now consider the more restrictive setting of a \textit{contractive} operator, formalized by the following additional assumption. Notably, we only assume the expected operator to be $q$-quasi-contractive, and not each realization.

\begin{assumption}[Problem Assumption, Contractiveness]\label{ass:contractive}
    We assume the operator $T$ is $q$-quasi-contractive for $q\in (0, 1)$ and that $\Fix(T)$ is nonempty.
\end{assumption}
We note that under Assumption \ref{ass:contractive}, $\Fix(T)$ is a singleton.

\begin{lemma}[Decrease Inequality]\label{lem:decrease_contr}
    Let Problem Assumptions \ref{ass:nonexp} and \ref{ass:contractive} and Variance Assumption \ref{ass:bounded_some_FP} hold, and let $(x_k)$ be the random iterates determined by Algorithm \ref{alg:SKM}. Then it holds that, with $\Fix(T)=\{p\}$ and for all $k\ge 0$, almost surely,
    \[
        \ExpD{k}{\|x_{k+1}-p\|^2}\le \left[(1-\lambda_k+\lambda_kq)^2+2\lambda_k^2\right]\|x_k-p\|^2+2\lambda_k^2\sigma_*^2.
    \]
\end{lemma}
\begin{proof}
    This follows by the Variance Transfer Lemma \ref{lem:transfer} and noting that
    \begin{align*}
        \ExpD{k}{\|x_{k+1}-p\|^2}&=\ExpD{k}{\|(1-\lambda_k)(x_k-p)+\lambda_k (Tx_k-p)+\lambda_k(T_{\xi_k}-T)x_k\|^2} \\
        &=\|(1-\lambda_k)(x_k-p)+\lambda_k (Tx_k-p)\|^2+\ExpD{k}{\|\lambda_k(T_{\xi_k}-T)x_k\|^2} \\
        &\le (1-\lambda_k+\lambda_kq)^2\|x_k-p\|^2+2\lambda_k^2\sigma_*^2+2\lambda_k^2\|x_k-p\|^2,
    \end{align*}
    which rewrites to the wanted inequality.
\end{proof}
\begin{theorem}[Strong Convergence of Iterates and Convergence Rates]
    Let Problem Assumptions \ref{ass:nonexp} and \ref{ass:contractive} and Variance Assumption \ref{ass:bounded_some_FP} hold, and let $(x_k)$ be the random iterates determined by Algorithm \ref{alg:SKM}. Then, for $\Fix(T)=\{p\}$,
    \begin{enumerate}
        \item Under Parameter Assumption \ref{ass:stepsize}, $(x_k)$ converges strongly to $p$ almost surely;
        \item Under the step-size schedule $\lambda_k=\frac{b}{k+a}$ for $a> b> 0$, it holds that, for $K\ge 1$
        {\small\[
            \Exp{\|x_K-p\|}\le \begin{dcases}
                \mathcal O(K^{-{b(1-q)}})&\text{if $2b<\tfrac{1}{1-q}$,} \\
                \mathcal O(\log(K)^{\frac12}K^{-\frac12})&\text{if $2b=\tfrac{1}{1-q}$,} \\
                \mathcal O(K^{-\frac12})&\text{if $2b>\tfrac{1}{1-q}$,}
            \end{dcases}\quad \text{and}\quad \Exp{\|Tx_K-x_K\|}\le \begin{dcases}
                \mathcal O(K^{-{b(1-q)}})&\text{if $2b<\tfrac{1}{1-q}$,} \\
                \mathcal O(\log(K)^{\frac12}K^{-\frac12})&\text{if $2b=\tfrac{1}{1-q}$,} \\
                \mathcal O(K^{-\frac12})&\text{if $2b>\tfrac{1}{1-q}$,}
            \end{dcases}
        \]}
    \end{enumerate}
\end{theorem}
\begin{proof}
    We prove the two statements separately.
    \begin{enumerate}
        \item By the Robbins-Siegmund Theorem \ref{thm:robbins_siegmund}, by rewriting the Decrease Inequality in Lemma \ref{lem:decrease_contr} as
        \[
            \ExpD{k}{\|x_{k+1}-p\|^2}\le \left[1+2\lambda_k^2\right]\|x_k-p\|^2-\left[1-(1-\lambda_k+\lambda_kq)^2\right]\|x_k-p\|^2+2\lambda_k^2\sigma_*^2,
        \]
        and noting that $1-(1-\lambda_k+\lambda_kq)^2\ge \lambda_k(1-q)$, we obtain that $\|x_k-p\|^2$ converges almost surely and that $\sum_{k=0}^\infty \lambda_k\|x_k-p\|^2<+\infty$ almost surely. As $\sum_{k=0}^\infty \lambda_k=+\infty$, we deduce that $\|x_k-p\|^2\to 0$ almost surely.
        \item This follows immediately by the Decrease Inequality in Lemma \ref{lem:decrease_contr} and Lemma \ref{lem:rates_contra}. The second inequality follows by noting that $\|Tx_K-x_K\|=\|Tx_K-p+p-x_K\|\le (1+q)\|x_K-p\|$. \qedhere
    \end{enumerate}
\end{proof}

\subsection{Numerical Illustration}\label{ssec:numerical}

We numerically validate our theoretical findings on a set of toy examples\footnote{The code is available on \url{https://github.com/DanielCortild/Stochastic-KM}.}. Specifically, we consider the setting of $2$ operators on $\H=\R^2$, $T_1\colon x\mapsto x$ the identity and $T_2\colon x\mapsto e^{i\theta}x$ a rotation by the angle $\theta\in (0, 2\pi)$, where we identified $\R^2$ with $\C$ for computational simplicity, sampled uniformly at random. As such, $T=\Exp{T_\xi}$ has as unique fixed point the origin. As this is a fixed point of both $T_1$ and $T_2$, Assumption \ref{ass:bounded_some_FP} holds with $\sigma_*^2=0$. Importantly, the uniformly bounded variance assumption fails for any $\theta\in (0, 2\pi)$, since
\[
    \Exp{\|(T-T_\xi)x\|^2}=\sin(\theta/2)^2\|x\|^2.
\]

\textbf{Experimental Setup.} We run the experiment on $36$ equally partitioned angles $\theta$ between $\pi/36$ and $\pi$, with initial point $x_0=(1,1)$. For each angle, given a trajectory, we compute $\ExpD{N_K}{\|Tx_{N_K}-x_{N_K}\|^2}$ and $\|Tx_K-x_K\|^2$, average it over multiple trajectories, and compare it to the result obtained in Theorems \ref{thm:residual} and \ref{thm:last}. These averages are over $100$ independent trajectories.

For each trajectory, we run $500000$ iterations. For the random expected residual, we set a relaxation parameter of $\lambda_k=\bar \lambda (k+1)^{-0.5}\ln(k+e)^{-1}$, and for the last-iterate expected residual, $\lambda_k=\bar\lambda (k+1)^{-2/3}$, for $\bar \lambda = 0.1$.

\textbf{Results.} The results are shown in Figure \ref{fig:numerical_validation}. The theoretical curves are slightly tighter than the ones obtained in the theorems, as we apply Remark \ref{rem:Young} for the comparison. The empirical decay is consistent with the dependence on $K$ predicted by Theorems \ref{thm:residual} and \ref{thm:last}. The theoretical bound is conservative in its constants, which is expected since Lemma \ref{lem:transfer} does not exploit the linear structure of this example. We note that the last-iterate bounds are overly conservative in many scenarios, but accurate (up to constant factors) in extreme cases.

\begin{figure}[H]
    \centering
    \includegraphics[width=1\linewidth]{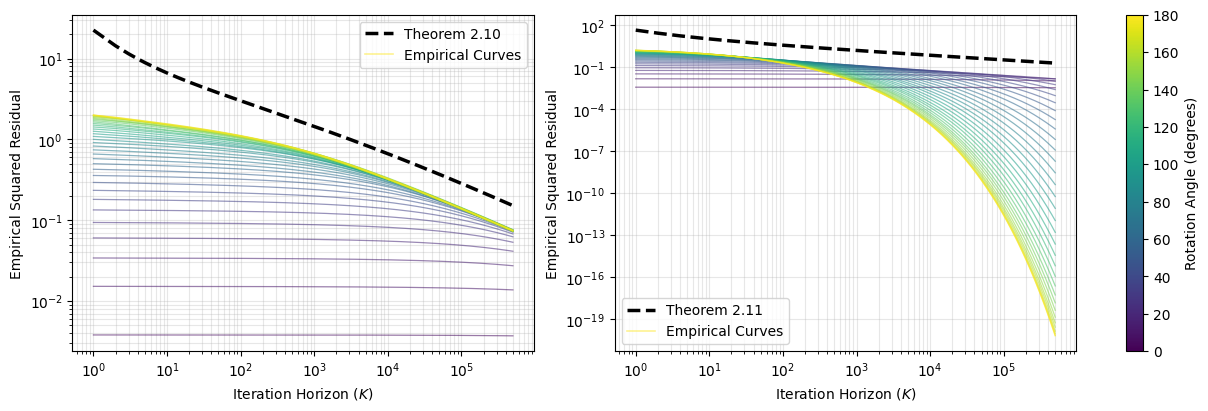}
    \caption{Empirical $\ExpD{N_K}{\|Tx_{N_K}-x_{N_K}\|^2}$ compared to Theorem \ref{thm:residual} (left) and empirical $\|Tx_K-x_K\|^2$ compared to Theorem \ref{thm:last} (right).}
    \label{fig:numerical_validation}
\end{figure}
\section{Special Cases}\label{sec:apps}

In this section, we illustrate the applicability of our results by considering three representative applications. First, in Section \ref{ssec:sgd}, we revisit Stochastic Gradient Descent, showing our framework recovers known guarantees. Secondly, in Section \ref{ssec:stos}, we analyze Stochastic Three-Operator Splitting and Stochastic Backward-Forward Splitting, for which we obtain, to the best of our knowledge, the first convergence guarantees without uniformly bounded variance. The convergence of Stochastic Three-Operator Splitting however requires the restrictive convex-graph assumption, which we overcome in Section \ref{ssec:sltos}, where we consider a novel Stochastic Lifted Three-Operator Splitting.

\subsection{Stochastic Gradient Descent}\label{ssec:sgd}

We consider the standard expectation minimization problem 
\begin{equation}\label{eq:Emin}\tag{$\E$-min}
    \min_{x\in \mathcal H} f(x)=\ExpD{\xi\sim\mathcal D}{f_\xi(x)},
\end{equation}
where $f_\xi\colon \H\to \R$ is a family of convex and $L$-smooth functions. We assume the problem is well-defined, in the sense that $\xi\mapsto f_\xi(x)$ is $\mathcal D$-measurable for all $x\in \H$ and that $\ExpD{\xi}{f_\xi(x)}$ is finite for all $x\in \H$, and that the problem is well-posed, in the sense that $f$ has a minimizer. Moreover, we assume the problem is differentiable, namely that $f_\xi$ is differentiable for $\D$-almost all $\xi$ and $\nabla f(x)=\ExpD{\xi}{\nabla f_\xi(x)}$. Regarding the noise, we make the weak assumption of bounded variance at the solution \cite{bach_nonasymptotic_2011}, namely that there exists a point $x_*\in \argmin f$ such that 
\begin{equation}\tag{SGD-Ass}\label{eq:assSGD}
    \Exp{\|\nabla f_\xi(x_*)\|^2}\le \sigma_*^2<+\infty.
\end{equation}
A natural algorithm is the one of \textit{Stochastic Gradient Descent} (SGD), given by, with $x_0\in \H$ and for $k\ge 0$,
\begin{equation}\label{eq:SGD}\tag{SGD}
    x_{k+1}=x_k-\gamma\lambda_k \nabla f_{\xi_k}(x_k), \quad \text{where}~ \xi_k\stackrel{\text{iid}}\sim \mathcal D,
\end{equation}
where $(\gamma\lambda_k)\subset (0, 2/L)$ is a sequence of \textit{step-sizes}. This algorithm may be cast into the form of Algorithm \ref{alg:SKM} by writing $T_\xi=I-\gamma \nabla f_\xi$, which is nonexpansive if $\gamma\in (0, 2/L)$, such that the Problem Assumption \ref{ass:nonexp} is verified. We note that $\Fix(T)=\{x_*\colon \nabla f(x_*)=0\}=\argmin f$. Moreover, for the $x_*\in \argmin f=\Fix(T)$ given in Equation \eqref{eq:assSGD},
\[
    \Exp{\|(T-T_\xi)x_*\|^2}=\gamma^2\Exp{\|\nabla f_\xi(x_*)-\nabla f(x_*)\|^2}=\gamma^2\Exp{\|\nabla f_\xi(x_*)\|^2}\le \gamma^2\sigma_*^2,
\]
so the Variance Assumption \ref{ass:bounded_some_FP} is verified up to a multiplicative factor of $\gamma^2$ in the constant.

As such, Corollary \ref{coro:complexity} and Theorem \ref{thm:weak_conv} imply the following.
\begin{corollary}[Application to Algorithm SGD]\label{coro:sgd}
    Consider Problem \eqref{eq:Emin} and assume the variance-at-solution condition \eqref{eq:assSGD} holds. Let $(x_k)$ be generated by the \eqref{eq:SGD} iterations with $\gamma\in (0, 2/L)$. Then, for $\lambda_0\in (0, 1)$,
    \[
        \Exp{\|\nabla f(x_{N_K})\|}\le \begin{cases}
            \mathcal O(K^{-\frac14}) &\text{if $\lambda_k\equiv \lambda_0K^{-\frac12}$}, \\
            \mathcal O(\ln(K)^{\frac12}K^{-\frac14}) &\text{if $\lambda_k= \lambda_0(k+1)^{-\frac12}\ln(k+e)^{-1}$},
        \end{cases}
    \]
    where $N_K$ is defined in \eqref{eq:def_prob}, and
    \[
        \Exp{\|\nabla f(x_{K})\|}\le \begin{cases}
            \mathcal O(K^{-\frac16}) &\text{if $\lambda_k\equiv \lambda_0K^{-\frac23}$}, \\
            \mathcal O(\ln(K)^{\frac12}K^{-\frac16}) &\text{if $\lambda_k= \lambda_0(k+1)^{-\frac23}$}.
        \end{cases}
    \]
    Moreover, for the non-finite-horizon step-size schedules, $(x_k)$ converges weakly almost surely to a minimizer of $f$, and $\|\nabla f(x_k)\|\to 0$ almost surely.
\end{corollary}

We make several remarks regarding this result.
\begin{enumerate}
    \item By $L$-smoothness, rates on the function value gap can be translated into rates on the expected gradient norm, since
    \[
        \Exp{\|\nabla f(x)\|}\le \sqrt{\Exp{\|\nabla f(x)\|^2}}\le \sqrt{2L\Exp{f(x)-\inf f}}.
    \]
    The rate $\mathcal O(K^{-1/4})$ on the expected gradient norm is the rate obtained by translating the known minimax rate for function values for stochastic approximation in the non-strongly convex case \cite{agarwal_informationtheoretic_2009}. Our obtained rate was already established under the same assumption in \cite{gower_sgd_2021,cortild_biasoptimal_2025} for the average iterate. 
    \item Our rate for the last iterate is slower than current state-of-the-art results \cite{attia_fast_2025,garrigos_lastiterate_2025}, which achieve (up to logarithmic factors) the rate of $\mathcal O(K^{-1/4})$. Notably, our rate matches the previously best-known rate, presented in \cite{bach_nonasymptotic_2011}.
\end{enumerate}
Overall, this result shows that classical guarantees for Stochastic Gradient Descent can be recovered as a consequence of our fixed point framework, without relying on problem-specific arguments.

\subsection{Stochastic Three-Operator Splitting}\label{ssec:stos}

We consider the monotone inclusion problem 
\begin{equation}\tag{MI3}\label{eq:MI}
    \text{Find $x_*\in \Zer(A+B+C)$},
\end{equation}
where $A$ and $B$ are maximally monotone operators, and $C$ is a $\tau$-cocoercive operator.
We assume we only have access to $C$ through an unbiased stochastic estimator, namely that $Cx=\Exp{C_\xi x}$, where $C_\xi$ is $\tau$-cocoercive $\D$-almost surely. We assume the problem is well-defined, in the sense that $\xi\mapsto C_\xi x$ is $\D$-measurable for all $x\in H$ and $\Exp{C_\xi x}$ is finite for all $x\in \H$, and well-posed, in the sense that $\Zer(A+B+C)\neq \emptyset$.

A standard algorithm to tackle Problem \eqref{eq:MI} is the \textit{Three-Operator Splitting} method \cite{davis_threeoperator_2017a}, which iterates an operator
\[
    T_{\det} = I-J_{\rho A}+J_{\rho B}\circ (2J_{\rho A}-I-\rho C\circ J_{\rho A}),
\]
where $\rho\in (0, 2\tau)$ acts as a step-size. We propose the following stochastic version of the operator
\begin{equation}\label{eq:TOS_xi}\tag{STOS-Op}
    T_\xi = I-J_{\rho A}+J_{\rho B}\circ (2J_{\rho A}-I-\rho C_\xi\circ J_{\rho A}).
\end{equation}
Applying Algorithm \ref{alg:SKM} with the operator $T_\xi$ gives rise to the \textit{Stochastic Three-Operator Splitting} method, given by, with $x_0\in \H$ and for $k\ge 0$,
\begin{equation}\tag{STOS}\label{eq:STOS}
    \begin{cases}
        y_k&=J_{\rho A}x_k,\\
        z_k&= J_{\rho B}(2y_k-x_k-\rho C_{\xi_k}y_k), \quad \text{where $\xi_k\stackrel{\text{iid}}\sim \mathcal D$,} \\
        x_{k+1}&=x_k-\lambda_k(y_k-z_k).
    \end{cases}
\end{equation}

An important realization is that in general $T_{\det}\neq \Exp{T_\xi}=T$, since $J_{\rho B}$ is in general a nonlinear operator. As such, we only show convergence to a point in $\Fix(T)$, which are not necessarily solutions to the original problem, provided this set is nonempty.

However, one can show that, in certain settings, the set of fixed points of $T=\Exp{T_\xi}$ coincides with the set of zeros of $A+B+C$, which are the solutions to our problem, such that the operator $T$ still encodes the wanted problem.

\begin{lemma}[Link between Fixed Points and Zeros]\label{lem:link}
    Let $A, B\colon \H\mapsto 2^\H$ be maximally monotone such that $\graph(B)$ is convex, and let $C_\xi\colon \H\to \H$ be $\tau$-cocoercive for $\mathcal D$-almost all $\xi$. Define $T_\xi$ as in Equation \eqref{eq:TOS_xi} and let $T=\Exp{T_\xi}$. Then $J_{\rho A}(\Fix(T))=\Zer(A+B+C)$.
\end{lemma}
\begin{proof}
    We show the equality by proving both inclusions.
    \begin{enumerate}
        \item[($\subseteq$)] Take $x\in \Fix(T)$, and define $v=J_{\rho A}x$, such that $\frac{x-v}{\rho}\in Av$. Moreover, define $w_\xi=J_{\rho B}(2v-x-\rho C_\xi v)$, such that 
        \[
            x=Tx=\Exp{x-v+w_\xi}\quad \iff \quad \Exp{w_\xi}=v.
        \]
        Moreover,
        \[
            u_\xi=\frac{v-x}{\rho}+\frac{v-w_\xi}{\rho}-C_\xi v\in Bw_\xi.
        \]
        As $B$ is assumed to have a convex graph, this implies that 
        \[
            \Exp{u_\xi}=\frac{v-x}{\rho}+\frac{v-\Exp{w_\xi}}{\rho}-\Exp{C_\xi v}=\frac{v-x}{\rho}+0-Cv\in B(\Exp{w_\xi})=Bv.
        \]
        As $\frac{x-v}{\rho}\in Av$, this yields the wanted conclusion.
        \item[($\supseteq$)] Take $z\in \Zer(A+B+C)$, and pick $a\in Az$ and $b\in Bz$ such that $a+b+Cz=0$. Define $x=z+\rho a$ such that $J_{\rho A}x=z$, and $w_\xi=J_{\rho B}(2z-x-\rho C_\xi z)$, such that 
        \[
            u_\xi=\frac{z-w_\xi}{\rho}+\frac{z-x}{\rho}-C_\xi z=\frac{z-w_\xi}{\rho}-a-C_\xi z\in Bw_\xi.
        \]
        By convexity of $\graph(B)$, it holds that 
        \[
            \Exp{u_\xi}=\frac{z-\Exp{w_\xi}}{\rho}-a-\Exp{C_\xi z}=\frac{z-\Exp{w_\xi}}{\rho}-a-Cz\in B(\Exp{w_\xi}).
        \]
        As $B$ is a monotone operator and $b\in Bz$, it holds that 
        \[
            \left\langle \frac{z-\Exp{w_\xi}}{\rho}-a-Cz-b, \Exp{w_\xi}-z\right\rangle\ge 0.
        \]
        Since $a+b+Cz=0$, this implies that $-\frac1\rho \|z-\Exp{w_\xi}\|^2\ge 0$, which in turn implies that $z=\Exp{w_\xi}$. This may be rearranged into $\Exp{T_\xi x}=x$, which implies that $x\in \Fix(T)$.
        \qedhere
    \end{enumerate}
\end{proof}

We make the assumption that the estimator has finite variance at a fixed-point, being that 
\begin{equation}\tag{MI3-Ass}\label{eq:MI-Var}
    \Exp{\|(C-C_\xi)J_{\rho A}p\|^2}\le \sigma_*^2<+\infty,
\end{equation}
where $p\in \Fix(T)$, and note that, provided $B$ has a convex graph, $J_{\rho A}p$ is a solution of Problem \eqref{eq:MI}.

We note that the operator $T_\xi$ is nonexpansive as each $C_\xi$ is $\tau$-cocoercive and $\rho\in (0,2\tau)$, by \cite[Proposition 3.1]{davis_threeoperator_2017a}. As such, the Problem Assumption \ref{ass:nonexp} is satisfied. Moreover, for some $p\in \Fix(T)$, 
\begin{align*}
    \Exp{\|(T_\xi-T)p\|^2}
    &= \inf_{y\in \H}\Exp{\|T_{\xi}p-y\|^2} \\
    &\le \Exp{\|T_{\xi}p-T_{\text{det}}p\|^2} \\
    &=\Exp{\|J_{\rho B}(2J_{\rho A}p-p-\rho C(J_{\rho A}p))-J_{\rho B}(2J_{\rho A}p-p-\rho C_\xi(J_{\rho A}p))\|^2} \\
    &\le \rho^2\Exp{\| (C_\xi-C)J_{\rho A}p\|^2} \\
    & \le \rho^2\sigma_*^2<+\infty,
\end{align*}
where the first inequality follows by nonexpansiveness of $J_{\rho B}$ and the second by assumption on $C_\xi$. As such, the Variance Assumption \ref{ass:bounded_some_FP} is also verified. 

\begin{corollary}[Application to Algorithm STOS]\label{coro:stos}
    Consider Problem \eqref{eq:MI} and assume that $\graph(B)$ is convex, and assume the variance-at-solution condition \eqref{eq:MI-Var} holds. Let $(x_k, y_k, z_k)$ be generated by the \eqref{eq:STOS} iterations with $\rho\in (0, 2\tau)$. Then, for $\lambda_0\in (0, 1)$,
    \[
        \Exp{\|Tx_{N_K}-x_{N_K}\|}\le \begin{cases}
            \mathcal O(K^{-\frac14}) &\text{if $\lambda_k\equiv \lambda_0K^{-\frac12}$}, \\
            \mathcal O(\ln(K)^{\frac12}K^{-\frac14}) &\text{if $\lambda_k= \lambda_0(k+1)^{-\frac12}\ln(k+e)^{-1}$},
        \end{cases}
    \]
    where $N_K$ is defined in \eqref{eq:def_prob}, and
    \[
        \Exp{\|Tx_{K}-x_K\|}\le \begin{cases}
            \mathcal O(K^{-\frac16}) &\text{if $\lambda_k\equiv \lambda_0K^{-\frac23}$}, \\
            \mathcal O(\ln(K)^{\frac12}K^{-\frac16}) &\text{if $\lambda_k= \lambda_0(k+1)^{-\frac23}$}.
        \end{cases}
    \]
    Moreover, for the non-finite-horizon step-size schedules, $(x_k)$ converges weakly almost surely to a point $p\in \Fix(T)$ such that $J_{\rho A}p\in \Zer(A+B+C)$, and $\|Tx_k-x_k\|\to 0$ almost surely.
\end{corollary}

To the best of our knowledge, only two prior works have considered the Stochastic Three-Operator Splitting algorithm, both in the functional setting. The first is \cite{yurtsever_stochastic_2016}, which considers the strongly convex setting, and hence does not offer a valid comparison. The second is \cite{yurtsever_three_2021}, which shows convergence in function values at a rate of $\mathcal O(K^{-1/2})$, which is not directly comparable to our residual terms. Both of these works operate in the uniformly bounded variance setting, and we are not aware of any prior works that do not make use of that assumption.

In Section \ref{ssec:sltos}, we study an algorithm overcoming the assumption of $\graph(B)$ being convex. To conclude this subsection, we analyze the case of $B=0$, where this assumption is automatically satisfied. Then the problem reduces to finding $x_*\in \Zer(A+C)$, where $A\colon \H\to 2^\H$ is a maximally monotone operator and $C\colon \H\to \H$ is a $\tau$-cocoercive operator, and Assumption \eqref{eq:MI-Var} reduces to finite variance of $C_\xi$ at a solution of the problem, by Lemma \ref{lem:link}. Moreover, Algorithm \eqref{eq:STOS} reduces to a stochastic version of the \textit{Backward-Forward splitting algorithm} \citep{attouch_backward_2018}, given by, with $x_0\in \H$ and for $k\ge 0$,
\begin{equation}\tag{SBF}\label{eq:SBF}
    x_{k+1}=(1-\lambda_k)x_k+\lambda_k(I-\rho C_{\xi_k})J_{\rho A}x_k, \quad \text{where $\xi_k\stackrel{\text{iid}}\sim \mathcal D$.}
\end{equation} Specifically, we obtain the following convergence result.
\begin{corollary}[Application to Algorithm SBF]\label{coro:sbf}
    Consider Problem \eqref{eq:MI} with $B=0$ and assume the variance-at-solution condition \eqref{eq:MI-Var} holds. Let $(x_k)$ be generated by the \eqref{eq:SBF} iterations with $\rho\in (0, 2\tau)$. Then, for $\lambda_0\in (0, 1)$,
    \[
        \Exp{\|Tx_{N_K}-x_{N_K}\|}\le \begin{cases}
            \mathcal O(K^{-\frac14}) &\text{if $\lambda_k\equiv \lambda_0K^{-\frac12}$}, \\
            \mathcal O(\ln(K)^{\frac12}K^{-\frac14}) &\text{if $\lambda_k= \lambda_0(k+1)^{-\frac12}\ln(k+e)^{-1}$},
        \end{cases}
    \]
    where $N_K$ is defined in \eqref{eq:def_prob}, and
    \[
        \Exp{\|Tx_{K}-x_K\|}\le \begin{cases}
            \mathcal O(K^{-\frac16}) &\text{if $\lambda_k\equiv \lambda_0K^{-\frac23}$}, \\
            \mathcal O(\ln(K)^{\frac12}K^{-\frac16}) &\text{if $\lambda_k= \lambda_0(k+1)^{-\frac23}$}.
        \end{cases}
    \]
    Moreover, for the non-finite-horizon step-size schedules, $(x_k)$ converges weakly almost surely to a point $p\in \Fix(T)$ such that $J_{\rho A}p\in \Zer(A+C)$, and $\|Tx_k-x_k\|\to 0$ almost surely.
\end{corollary}
This result is, to the best of our knowledge, the first result on a stochastic version of the Backward-Forward splitting algorithm under our variance assumption. Uniform boundedness assumptions have been studied in this context previously \cite{tran-dinh_vfosa_2025}. Stochastic Forward-Backward splitting, based on the opposite operator order, was studied in \cite{combettes_stochastic_2016} under different noise assumptions.

\subsection{Stochastic Lifted Three-Operator Splitting}\label{ssec:sltos}

To overcome the strong assumption of convex graph of $B$ in Section \ref{ssec:stos}, we introduce a lifted version of the algorithm. Observe that a point $x\in \H$ solves Problem \eqref{eq:MI} if, and only if, there exists a point $u\in \H$ such that 
\begin{equation}\tag{MI3-Lifted}\label{eq:mi3-lifted}
    \begin{pmatrix}
        x \\ u
    \end{pmatrix} \in \Zer\left(\underbrace{\begin{pmatrix}
        A & 0 \\ 0 & B^{-1}
    \end{pmatrix}}_{\mathcal A}+\underbrace{\begin{pmatrix}
        0 & I \\ -I & 0
    \end{pmatrix}}_{\mathcal B}+\underbrace{\begin{pmatrix}
        C & 0 \\ 0 & 0
    \end{pmatrix}}_{\mathcal C}\right).
\end{equation}
For $A, B$ maximally monotone operators, $B^{-1}$, and therefore also $\mathcal A$, is maximally monotone \cite{bauschke_convex_2017}. Moreover, $\mathcal B$ is linear, continuous and skew-symmetric, and therefore also maximally monotone. Finally, $\mathcal C$ is $\tau$-cocoercive since $C$ is. Moreover, without additional assumptions, $\mathcal B$ is linear, and therefore has a convex graph \cite{alves_maximal_2009}. 

Define $\mathcal C_\xi$ as the matrix with upper-left entry $C_\xi$, the stochastic evaluation of $C$, and $0$ everywhere else. The Stochastic Three-Operator Splitting operator from Section \ref{ssec:stos} applied to Problem \eqref{eq:mi3-lifted} then yields
\[
    \mathcal T_\xi=I-J_{\rho\mathcal A}+J_{\rho B}\circ (2J_{\rho \mathcal A}-I-\rho \mathcal C_{\xi}\circ J_{\rho \mathcal A}).
\]
Lemma \ref{lem:link} guarantees that, for $\mathcal T=\Exp{\mathcal T_\xi}$, it holds that $J_{\rho \mathcal A}(\Fix(\mathcal T))=\Zer(\mathcal A+\mathcal B+\mathcal C)$. Specifically, for any $(\hat v, \hat u)\in \Fix(\mathcal T)$, it holds that $J_{\rho\mathcal A}(\hat v, \hat u)=(J_{\rho A}\hat v, J_{\rho B^{-1}}\hat u)\in \Zer(\mathcal A+\mathcal B+\mathcal C)$, so that $J_{\rho A}\hat v\in \Zer(A+B+C)$, which thus solves Problem \eqref{eq:MI}.

Applying Algorithm \ref{alg:SKM} with the operator $\mathcal T_\xi$ gives rise to the \textit{Stochastic Lifted Three-Operator Splitting} method, given by, with $x_0=(v_0, u_0)\in \H\times \H$ and for $k\ge 0$,
\begin{equation}\tag{SLTOS}\label{eq:sltos}
    \begin{cases}
        (y_{k}^v, y_{k}^u) &= (J_{\rho A}(v_k), u_k-\rho J_{\rho^{-1}B}(u_k/\rho)), \\
        (z_{k}^v, z_{k}^u) &= (2y_{k}^v-v_k-\rho C_{\xi_k}y_k^v, 2y_k^u-u_k), \quad \text{where $\xi_k\stackrel{\text{iid}}\sim \mathcal D$,} \\
        (w_k^v, w_k^u) &= (1+\rho^2)^{-1}(z_k^v-\rho z_k^u, \rho z_k^v+z_k^u), \\
        (v_{k+1}, u_{k+1}) &= (v_k-\lambda_k(y_k^v-w_k^v), u_k-\lambda_k(y_k^u-w_k^u)).
    \end{cases}
\end{equation}
Note that we used the relation $J_{\rho B^{-1}}(x)=x-\rho J_{\rho^{-1}B}(x/\rho)$. 

Finally, Assumption \eqref{eq:MI-Var} reduces to assuming there exists a solution $\hat x$ to Problem \eqref{eq:MI} so that 
\begin{equation}\tag{Lifted-MI3-Ass}\label{eq:Lifted-MI-Var}
    \Exp{\|(C-C_\xi)\hat x\|^2}\le \sigma_*^2<+\infty,
\end{equation}
which, by the same reasoning as in Section \ref{ssec:stos}, yields the Variance Assumption \ref{ass:bounded_some_FP}. Specifically, we obtain the following result.
\begin{corollary}[Application to Algorithm SLTOS]\label{coro:sltos}
    Consider Problem \eqref{eq:MI} and assume the variance-at-solution condition \eqref{eq:Lifted-MI-Var} holds. Let $(x_k)$ be generated by the \eqref{eq:sltos} iterations with $\rho\in (0, 2\tau)$. Then, for $\lambda_0\in (0, 1)$,
    \[
        \Exp{\|\mathcal Tx_{N_K}-x_{N_K}\|}\le \begin{cases}
            \mathcal O(K^{-\frac14}) &\text{if $\lambda_k\equiv \lambda_0K^{-\frac12}$}, \\
            \mathcal O(\ln(K)^{\frac12}K^{-\frac14}) &\text{if $\lambda_k= \lambda_0(k+1)^{-\frac12}\ln(k+e)^{-1}$},
        \end{cases}
    \]
    where $N_K$ is defined in \eqref{eq:def_prob}, and
    \[
        \Exp{\|\mathcal Tx_{K}-x_K\|}\le \begin{cases}
            \mathcal O(K^{-\frac16}) &\text{if $\lambda_k\equiv \lambda_0K^{-\frac23}$}, \\
            \mathcal O(\ln(K)^{\frac12}K^{-\frac16}) &\text{if $\lambda_k= \lambda_0(k+1)^{-\frac23}$}.
        \end{cases}
    \]
    Moreover, for the non-finite-horizon step-size schedules, $(x_k)$ converges weakly almost surely to a point $\hat x=(\hat v, \hat u)\in \Fix(\mathcal T)$ satisfying $J_{\rho A}(\hat v)\in \Zer(A+B+C)$, and $\|\mathcal Tx_k-x_k\|\to 0$ almost surely.
\end{corollary}
\section{Conclusion}

In this paper, we established the convergence of the Stochastic Krasnoselskii-Mann iterations for expected fixed-point problems under the weak assumption of finite variance at a single fixed point, relaxing the typically-made uniformly bounded variance assumption. Under this weaker assumption, we recover the best-known $\mathcal O(\varepsilon^{-4})$ stochastic oracle complexity for the convergence of a random expected residual, show the rate holds almost surely for the running minimum residual, derive a $\mathcal O(\varepsilon^{-6})$ oracle complexity for the last iterate expected residual, and prove almost sure weak convergence of the iterates to a fixed point. These results show that uniformly bounded variance is not intrinsic in the analysis of stochastic fixed point methods, and can be replaced by a significantly weaker local condition.

Furthermore, we demonstrate the applicability of our results by considering four representative examples. For the case of Stochastic Gradient Descent, we recover known convergence results. However, for Stochastic Three-Operator Splitting and Stochastic Backward-Forward Splitting, we obtain results which are the first to bypass the uniformly bounded variance assumption. We additionally introduce a novel Stochastic Lifted Three-Operator Splitting, which removes the convex graph assumption required by our Stochastic Three-Operator Splitting method.

\paragraph{Outlook.} An interesting direction of future research is to understand how the stochasticity in Algorithm \ref{alg:SKM} interacts with inertia. Specifically, in order to control the stochasticity, we made the assumption that the relaxation parameters decrease to $0$. However, for inertial results, it is typical that the relaxation sequences need to be uniformly bounded away from $0$ \cite{cortild_krasnoselskii_2025,maulen_inertial_2024}.

\paragraph{Acknowledgments.} 
Daniel Cortild acknowledges the support of the Clarendon Funds Scholarship. 
Coralia Cartis' work was supported by the Hong Kong Innovation and Technology Commission (InnoHK Project CIMDA) and by the EPSRC grant EP/Y028872/1, Mathematical Foundations of Intelligence: An “Erlangen Programme” for AI. 
Juan Peypouquet benefited from the support of the FMJH Program Gaspard Monge for optimization and operations research and their interactions with data science.




\appendix

\section{Auxiliary Results}\label{sec:aux}

This section collects some known facts about real sequences and random variables, which are central to our convergence results. The proofs are included for the reader's convenience.

We start by recalling the well-known Robbins-Siegmund Almost-Supermartingale Theorem \cite{robbins_convergence_1971}.

\begin{theorem}[Robbins-Siegmund Almost-Supermartingale Theorem]\label{lem:robbins_siegmund}\label{thm:robbins_siegmund}
    Let $(X_k), (Y_k)$ and $(Z_k)$ be nonnegative random variables adapted to the filtration $\mathcal F_k$. Let $\eta_k$ be a nonnegative real sequence such that $\sum_{k=0}^\infty\eta_k<+\infty$. Assume that, for every $k\ge 0$, almost surely, 
    \[
        \Exp{Y_{k+1}|\mathcal F_k}\le (1+\eta_k)Y_k-X_k+Z_k\quad \text{and}\quad \sum_{k=0}^\infty Z_k<+\infty,
    \]
    Then we have that $\sum_{k=0}^\infty X_k<+\infty$ almost surely and that $Y_k$ converges almost surely.
\end{theorem}

\begin{lemma}\label{lem:small_o}
    Let $(x_k), (\eta_k)$ be real nonnegative sequences such that 
    \[
        \sum_{k=0}^\infty \eta_kx_k<+\infty\quad\text{and}\quad \sum_{k=0}^\infty \eta_k=+\infty.
    \]
    Then, as $K\to \infty$,
    \[
        \min_{k=0, \ldots, K-1}x_k= o\left(\frac{1}{\sum_{k=0}^{K-1}\eta_k}\right).
    \]
\end{lemma}
\begin{proof}
    Define $S_K=\sum_{k=0}^{K-1}\eta_k$ and $m_K=\min_{k=0, \ldots, K-1}x_k$. Fix a $K\ge 1$ and a $N< K$. Then it holds, for all $k=N, \ldots, K-1$, that $m_K\le x_k$. Specifically, 
    \[
        (S_K-S_N)m_K=\sum_{k=N}^{K-1}\eta_km_K\le \sum_{k=N}^{K-1}\eta_kx_k,
    \]
    which implies that 
    \[
        S_Km_K\le \frac{S_K}{S_K-S_N}\cdot \sum_{k=N}^{K-1}\eta_kx_k=\frac{1}{1-\frac{S_N}{S_K}}\cdot \sum_{k=N}^{K-1}\eta_kx_k.
    \]
    Taking the $\limsup$ on both sides yields that 
    \[
        \limsup_{K\to \infty}S_Km_K\le \sum_{k=N}^{\infty}\eta_kx_k,
    \]
    since $S_K\to \infty$. As $\sum_{k=0}^\infty \eta_kx_k<+\infty$, taking the limit $N\to \infty$ yields that $\limsup_{K\to\infty}S_Km_K\le 0$, which means that $m_K=o(1/S_K)$, as wanted. 
\end{proof}

Finally, we will make use of \cite[Lemma 1]{bremen79_almost_2020}, which was originally used to show almost sure last-iterate convergence of SGD.

\begin{lemma}\label{lem:to_zero}
    Let $(r_k)$ and $(\eta_k)$ be nonnegative sequences in $\R$ such that $\sum_{k=0}^\infty \eta_kr_k^2<+\infty$ and that $\sum_{k=0}^\infty\eta_k=+\infty$. Assume there exists an $L\ge 0$ such that, for all $K\ge 0$ and $\tau\ge 0$,
    \[
        |r_{K+\tau}-r_K|\le L\cdot \left(\sum_{k=K}^{K+\tau-1}\eta_kr_k+\left\|\sum_{k=K}^{K+\tau-1}\eta_kw_k\right\|\right),
    \]
    where $(w_k)$ is a sequence in $\H$ such that $(\sum_{k=0}^K\eta_kw_k)_K$ converges in $\H$. Then $(r_k)$ converges to $0$.
\end{lemma}
\begin{proof}
    If $L=0$, the condition implies $r_k\equiv r$ is constant, and therefore $r_k\equiv 0$, such that the result follows. Now assume $L>0$.
    
    As $\sum_{k=0}^\infty\eta_k=+\infty$ and $\sum_{k=0}^\infty \eta_kr_k^2<+\infty$, it must hold that $\liminf_{k\to\infty}r_k=0$. If $\limsup_{k\to\infty}r_k>0$, there must exist a $\varsigma>0$ such that $r_k>3\varsigma$ infinitely often. Define the sequences $(n_k)$ and $(m_k)$ such that 
    \[
        n_k<m_k<n_{k+1}\quad \text{for all $k\ge 0$}, \quad \text{and} \quad \begin{dcases}
            r_{i}\le \varsigma\quad & \text{if $n_k\le i< m_k$} \\
            r_{i}> \varsigma\quad & {\text{if $m_k\le i<n_{k+1}$}}.
        \end{dcases}
    \]
    Since $\sum_{k=0}^\infty\eta_kr_k^2<+\infty$ and $(\sum_{k=0}^K\eta_kw_k)_K$ converges, for any $\varepsilon>0$, there must exist $\bar j\ge 0$ such that, for all $M\ge m_{\bar j}$, 
    \[
        \sum_{k=m_{\bar j}}^M\eta_kr_k^2< \varepsilon\quad \text{and}\quad \left\|\sum_{k=m_{\bar j}}^M\eta_kw_k\right\|<\varepsilon.
    \]
    Now fix any $m\in [m_j, n_{j+1})$ for some $j\ge \bar j$, such that 
    \begin{align*}
        |r_{m}-r_{n_{j+1}}|
        &\le L\cdot \left(\sum_{k=m}^{n_{j+1}-1}\eta_kr_k+\left\|\sum_{k=m}^{n_{j+1}-1}\eta_kw_k\right\|\right) \\
        &\le L\cdot \left(\varsigma^{-1}\sum_{k=m}^{n_{j+1}-1}\eta_kr_k \varsigma+\varepsilon\right) \\
        &\le L\cdot \left(\varsigma^{-1}\sum_{k=m}^{n_{j+1}-1}\eta_kr_k^2+\varepsilon\right) \\
        &\le L\cdot \left(\varsigma^{-1}\varepsilon+\varepsilon\right).
    \end{align*}
    Picking $\varepsilon=\frac{\varsigma}{L(\varsigma^{-1}+1)}$, we get that 
    \[
        r_m\le r_{n_{j+1}}+|r_m-r_{n_{j+1}}|\le 2\varsigma.
    \]
    This contradicts the assumption that $r_k>3\varsigma$ infinitely often, and hence means that it must hold that $\limsup_{k\to\infty}r_k=0$, as claimed.
\end{proof}

\begin{lemma}\label{lem:rates_contra}
    Let $a>0, b, c, d\ge 0$ and $(u_k)$ be a nonnegative sequence in $\R$ so that, for every $k\ge 0$, 
    \[
        u_{k+1}\le \left(1-\frac{b}{k+a}+\frac{c}{(k+a)^2}\right)u_k+\frac{d}{(k+a)^2}.
    \]
    Then 
    \[
        u_k\le \begin{dcases}
            \mathcal O(k^{-b})&\text{if $b<1$,} \\
            \mathcal O(\log(k)k^{-1})&\text{if $b=1$,} \\
            \mathcal O(k^{-1})&\text{if $b>1$.}
        \end{dcases}
    \]
\end{lemma}
\begin{proof}
    Choose $k_0\ge 0$ sufficiently large so that $1-\frac{b}{k+a}+\frac{c}{(k+a)^2}> 0$ for $k\ge k_0$. Define 
    \[
        P_{k, j}=\prod_{i=j}^{k-1}\left(1-\frac{b}{i+a}+\frac{c}{(i+a)^2}\right),
    \]
    so that
    \[
        u_k\le P_{k, k_0}u_{k_0}+d\cdot \sum_{j=k_0}^{k-1}\frac{P_{k, j+1}}{(j+a)^2}.
    \]
    Now note that, as $\log(1+x)\le x$ for $x>-1$, for $j\ge k_0$,
    \[
        \log P_{k, j}\le -b\sum_{i=j}^{k-1}\frac{1}{i+a}+c\sum_{i=j}^{k-1}\frac{1}{(i+a)^2}\le -b\log\left(\frac{k+a}{j+a}\right)+\frac{c\pi^2}{6},
    \]
    so that
    \[
        P_{k, j}\le \exp\left(\frac{c\pi^2}{6}\right)\cdot \left(\frac{j+a}{k+a}\right)^{b}=M\cdot \left(\frac{j+a}{k+a}\right)^{b}.
    \]
    Plugging this back into the recursion gives
    \[
        u_k\le M\cdot \left(\frac{k_0+a}{k+a}\right)^{b}u_{k_0}+dM\cdot \sum_{j=k_0}^{k-1}\frac{\left(\frac{j+1+a}{k+a}\right)^{b}}{(j+a)^2}\le M\cdot \left(\frac{k_0+a}{k+a}\right)^{b}u_{k_0}+\frac{dM}{(k+a)^b}\left(1+\frac1a\right)^b\cdot \sum_{j=k_0}^{k-1}(j+a)^{b-2},
    \]
    which yields the wanted result by considering the cases $b<1$, $b=1$ and $b>1$.
\end{proof}

{
\small
\bibliographystyle{abbrv}
\bibliography{references}
}

\end{document}